\documentclass{amsart}
\usepackage{amsmath}
\usepackage{amssymb}
\usepackage{amsfonts}
\usepackage{float}
\usepackage{fancyhdr}
\usepackage[noadjust]{cite}

\newtheorem{theorem}{Theorem}[section]
\theoremstyle{plain}

\newtheorem{corollary}[theorem]{Corollary}
\newtheorem{lemma}[theorem]{Lemma}
\newtheorem{definition}[theorem]{Definition}
\newtheorem{proposition}[theorem]{Proposition}
\newtheorem{remark}[theorem]{Remark}
\numberwithin{equation}{section}
\input{tcilatex}
\def\tbigcup{\mathop{\textstyle \bigcup }}
\def\tsum{\mathop{\textstyle \sum }}

\begin{document}
\title{Set-Direct Factorizations of Groups}
\author{Dan Levy}
\address[Dan Levy]{The School of Computer Sciences \\
The Academic College of Tel-Aviv-Yaffo \\
2 Rabenu Yeruham St.\\
Tel-Aviv 61083\\
Israel}
\email{danlevy@mta.ac.il}
\author{Attila Mar\'{o}ti}
\address[Attila Mar\'{o}ti]{Alfr\'{e}d R\'{e}nyi Institute of Mathematics,
Hungarian Academy of Sciences\\
Re\'{a}ltanoda utca 13-15\\
H-1053, Budapest\\
Hungary}
\email{maroti.attila@renyi.mta.hu}
\thanks{A.M. was supported by the National Research, Development and
Innovation Office (NKFIH) Grant No.~K115799 and Grant No.~ERC$\_$HU$\_$15
118286, and by the J\'anos Bolyai Research Scholarship of the Hungarian
Academy of Sciences. The work of A.M. on the project leading to this
application has received funding from the European Research Council (ERC)
under the European Union's Horizon 2020 research and innovation programme
(grant agreement No 741420). He also received funding from ERC 648017.}
\date{\today }
\subjclass[2000]{ 20K25, 20D40}
\keywords{direct factorizations of groups, central products}

\begin{abstract}
We consider factorizations $G=XY$ where $G$ is a general group, $X$ and $Y$
are normal subsets of $G$ and any $g\in G$ has a unique representation $g=xy$
with $x\in X$ and $y\in Y$. This definition coincides with the customary and
extensively studied definition of a direct product decomposition by subsets
of a finite abelian group. Our main result states that a group $G$ has such
a factorization if and only if $G$ is a central product of $\left\langle
X\right\rangle $ and $\left\langle Y\right\rangle $ and the central subgroup 
$\left\langle X\right\rangle \cap \left\langle Y\right\rangle $ satisfies
certain abelian factorization conditions. We analyze some special cases and
give examples. In particular, simple groups have no non-trivial set-direct
factorization.
\end{abstract}

\maketitle

\section{Introduction}

Factorizations of groups is an important topic in group theory that has many
facets. The most basic and best studied type of factorization is the direct
product factorization of a group into two normal subgroups. If $G$ is a
group and $H$ and $K$ are two normal subgroups of $G$ then $G=H\times K$ if
and only if each $g\in G$ has a unique representation $g=hk$ with $h\in H$
and $k\in K$. One possibility to generalize this definition is to relax the
condition that both $H$ and $K$ are normal. This leads to the well-known
concept of a semi-direct product of subgroups (only one of the factors is
assumed to be normal) and also to the consideration of factorizations $G=HK$
where neither of the two subgroups $H$ and $K$ is normal, and even the
unique representation condition is not necessarily assumed. To get a glimpse
of the possibilities see the seminal classification result in \cite%
{LiebeckPraegerSaxl} of maximal decompositions $G=HK$ where $G$ is a finite
simple group and $H$ and $K$ are two maximal subgroups of $G$.

Another generalization arose from a geometry problem of Minkowski \cite%
{MinkowskiBook1896}. In 1938 Haj\'{o}s \cite{Hajos1938} reformulated this
problem as an equivalent factorization problem of a finite abelian group,
where the factors need not be subgroups. More precisely, if $G$ is an
abelian group written additively then a (set) factorization of $G$ is a
representation of $G$ in the form $G=H+K$ where $H$ and $K$ are subsets of $%
G $ and for each $g\in G$ there is a unique pair $h\in H$ and $k\in K$ such
that $g=h+k$. Four years later, Haj\'{o}s \cite{Hajos1942} solved
Minkowski's problem by solving the equivalent group factorization problem.
This initiated the investigation of set factorizations of abelian groups.
The interested reader is referred to the book \cite{SzaboSandsBook2009} by
Szab\'{o} and Sands for a comprehensive account of problems, techniques,
results and applications of this field.

In the present paper we consider set factorizations of a general group. For
abelian groups our definition coincides with the one described above. Some
specialized results which apply to our definition appeared in \cite%
{KoradiHermannSzabo2001}. Another related concept which is studied in the
literature, is the concept of a logarithmic signature. This concept found
use in the search for new cryptographic algorithms which are based on finite
non-abelian groups, and this application motivates the bulk of the available
results. A logarithmic signature of a group $G$ is a sequence $\left[ \alpha
_{1},...,\alpha _{s}\right] $ of ordered subsets $\alpha _{i}\subseteq G$
(not necessarily normal) such that each $g\in G$ has a unique representation 
$g=g_{1}\cdots g_{s}$ with $g_{i}\in \alpha _{i}$ for all $1\leq i\leq s$.
The first proposal for a cryptosystem which is based on logarithmic
signatures can be found in \cite{Mag86}. Results on short logarithmic
signatures can be found in \cite{LempkenTran2005}.

\begin{definition}
\label{Def_DirectProductOfSubsets}Let $G$ be a group, and let $X$ and $Y$ be
two non-empty subsets of $G$. We shall say that the setwise product $XY$ is
direct, and denote this fact by writing $X\times Y$ for the set $XY$, if
both $X$ and $Y$ are normal in $G$, and if every $g\in XY$ has a unique
representation $g=xy$ with $x\in X$ and $y\in Y$.
\end{definition}

We shall say that the group $G$ has a \textit{set-direct factorization} (%
\textit{decomposition}) or is a \textit{set-direct product}, if $G=X\times Y$
for some $X,Y\subseteq G$. A set-direct factorization of $G$ will be called
non-trivial, if and only if none of $X$ or $Y$ is a singleton consisting of
a central element (whereby the second factor must be $G$). Furthermore,
following \cite[p.5]{SzaboSandsBook2009}, we shall say that a subset $X$ of
a group $G$ is \textit{normalized} if $1_{G}\in X$, and that the set-direct
factorization $G=X\times Y$ is \textit{normalized} if both $X$ and $Y$ are
normalized. We shall show (Remark \ref{Rem_Normalization}) that $Z\left(
G\right) \times Z\left( G\right) $ acts on the set of all direct
factorizations of $G$ and that each orbit of this action contains at least
one normalized factorization.

Our main result is a necessary and sufficient condition for a group $G$ and
an unordered pair $X,Y$ of normal subsets of $G$ to satisfy $G=X\times Y$.
This condition involves a certain central subgroup $Z$ of $G$, and a family
of set-direct factorizations of $Z$. Recall (see Section \ref%
{Subsect_CentralProducts}) that a group $G$ is a central product of two
subgroups $M$ and $N$ if $G=MN$, and $M$ and $N$ centralize each other. In
this case $M$ and $N$ are normal in $G$ and $Z:=M\cap N$ is central in $G$.
We write $G=M\circ _{Z}N$.

\begin{theorem}
\label{Th_MainNecessaryAndSufficient}Let $G$ be a group and let $X,Y$ be
normal subsets of $G$. Set $M=\left\langle X\right\rangle $, $N=\left\langle
Y\right\rangle $, $Z:=M\cap N$. For every $m\in M$ set $X_{m}:=\left(
m^{-1}X\right) \cap Z$ and for every $n\in N$ set $Y_{n}:=\left(
n^{-1}Y\right) \cap Z$. Then $G=X\times Y$ if and only if the following two
conditions are met.

\begin{enumerate}
\item[(a)] $G=M\circ _{Z}N$.

\item[(b)] $Z=X_{m}\times Y_{n}$ for every $m\in M$ and $n\in N$.
\end{enumerate}
\end{theorem}

\begin{corollary}
Simple groups have no non-trivial set-direct factorization.
\end{corollary}

Condition (b) of Theorem \ref{Th_MainNecessaryAndSufficient} specifies a
family of set-direct factorizations of $Z$, and the next definition
characterizes the families of set-direct factorizations of $Z$ that arise in
this way. Recall that if $H$ is an abelian group, and $S\subseteq H$ then
the kernel of $S$ in $H$ is defined by $K_{H}\left( S\right) :=\left\{ h\in
H|hS=S\right\} $ (we also write $K\left( S\right) $ if $H$ is clear from the
context). One can easily check that $K_{H}\left( S\right) $ is a subgroup of 
$H$.

\begin{definition}
\label{Def_FactorizationSystem}Let $Z$ be an abelian group and let $\mathcal{%
M}=\left\{ M_{i}\right\} _{i\in I}$ and $\mathcal{N}=\left\{ N_{j}\right\}
_{j\in J}$ be two multisets (i.e., repetitions allowed) of subgroups of $Z$.
An $\mathcal{MN}$-direct factorization system of $Z$ is a pair of multisets $%
\mathcal{A}=\left\{ A_{i}\right\} _{i\in I}$ and $\mathcal{B}=\left\{
B_{j}\right\} _{j\in J}$ of subsets of $Z$ which satisfy for all $i\in I$
and $j\in J$ 
\begin{equation*}
Z=A_{i}\times B_{j}\text{, }M_{i}\leq K_{Z}\left( A_{i}\right) \text{, }%
N_{j}\leq K_{Z}\left( B_{j}\right) \text{.}
\end{equation*}
\end{definition}

To make the connection between Condition (b) of Theorem \ref%
{Th_MainNecessaryAndSufficient} and Definition \ref{Def_FactorizationSystem}%
, we need the following fact (see Section \ref%
{SubSect_CentralMultiplicationAction}, Lemma \ref{Lem_ZAction}). Let $G$ be
a group, $N\trianglelefteq G$, and $Z\leq N$ a central subgroup of $G$. Then 
$Z$ acts by multiplication on the set $\Omega _{N}$ of all conjugacy classes
of $G$ contained in $N$.

\begin{theorem}
\label{Th_XxYAndFactorizationSystems}Let $G$ be a group and let $X,Y$ be
normal subsets of $G$. Set $M=\left\langle X\right\rangle $, $N=\left\langle
Y\right\rangle $, $Z:=M\cap N$, and assume that $G=M\circ _{Z}N$. For each $%
m\in M$ and $n\in N$ let $M_{m}$ and $N_{n}$ be, respectively, the
stabilizers of the conjugacy classes of $m$ and $n$ with respect to the
multiplication action of $Z$. Set $\mathcal{M}=\left\{ M_{m}\right\} _{m\in
M}$ and $\mathcal{N}=\left\{ N_{n}\right\} _{n\in N}$. Then $G=X\times Y$ if
and only if the pair of multisets $\left( \left\{ X_{m}\right\} _{m\in
M},\left\{ Y_{n}\right\} _{n\in N}\right) $ is an $\mathcal{MN}$-direct
factorization system of $Z$, where for every $m\in M$ and for every $n\in N$%
, we set $X_{m}:=\left( m^{-1}X\right) \cap Z$ and $Y_{n}:=\left(
n^{-1}Y\right) \cap Z$.
\end{theorem}

The next theorem shows that starting from a central product $G=M\circ _{Z}N$
and an appropriate set-direct factorization system of $Z$, one can obtain a
set-direct decomposition of $G$. We denote by $O\left( \Omega _{M}\right) $
and $O\left( \Omega _{N}\right) $ the set of all \textit{orbits} of the
multiplication action of $Z$ on $\Omega _{M}$ and $\Omega _{N}$
respectively. Note (Lemma \ref{Lem_ZAction} part 3) that the stabilizer of
an orbit is the stabilizer of any conjugacy class belonging to the orbit.

\begin{theorem}
\label{Th_ConstructingXxYFromFacSys}Let $G=M\circ _{Z}N$. Set $I:=O\left(
\Omega _{M}\right) $ and $J:=O\left( \Omega _{N}\right) $. For each $i\in I$
and $j\in J$ let $M_{i}$ and $N_{j}$ be the stabilizers of the orbits $i$
and $j$. Set $\mathcal{M}:=\left\{ M_{i}\right\} _{i\in I}$ and $\mathcal{N}%
:=\left\{ N_{j}\right\} _{j\in J}$. Assume that $\mathcal{A}:=\left\{
A_{i}\right\} _{i\in I}$ and $\mathcal{B}:=\left\{ B_{j}\right\} _{j\in J}$
is an $\mathcal{MN}$-direct factorization system of $Z$. For each $i\in I$
and $j\in J$ fix conjugacy classes $C_{i}$ and $D_{j}$ belonging to the
orbits $i$ and $j$ respectively. Then $G=X\times Y$ is a set-direct
decomposition of $G$, where 
\begin{equation*}
X:=\tbigcup\limits_{i\in I}A_{i}C_{i}\text{ and }Y:=\tbigcup\limits_{j\in
J}B_{j}D_{j}\text{.}
\end{equation*}
\end{theorem}

Thus, the set-direct factorizations of a general group $G$ can be obtained,
in principle, from the knowledge of the central subgroups of $G$, the
central product decompositions of $G$ in which they are involved (each
central $Z\leq G$ is involved in at least one such decomposition - $G=G\circ
_{Z}Z$), the stabilizers of their multiplication action on the set of
conjugacy classes of $G$ contained in the factors of the central products,
and the associated factorization systems.

Now we consider two special cases of this general observation. The first
case is when one of the factors is a group (and then the second factor is a
normal transversal for this group). We shall say, given a group $G$, that $%
z\in Z\left( G\right) $ is \emph{semi-regular} if $zC\neq C$ for every
conjugacy class $C$ of $G$. In the following $k\left( G\right) $ denotes the
number of conjugacy classes of $G$.

\begin{theorem}
\label{Th_XhasNormalTransversal}Let $G$ be a group and let $X$ and $Y$ be
normal subsets of $G$. Suppose that $X$ is a group. Set $N:=\langle Y\rangle 
$ and $Z:=X\cap N$, and for each $n\in N$ set $Y_{n}:=\left( n^{-1}Y\right)
\cap Z$. Then the following two conditions are equivalent:

\begin{enumerate}
\item[(1)] $G=X\times Y$.

\item[(2)] $G=X\circ _{Z}N$ and $\left\vert Y_{n}\right\vert =1$ for all $%
n\in N$.
\end{enumerate}

Furthermore, if either of (1) or (2) holds, $Y$ is a normal transversal of $%
Z $ in $N$ and $Z$ acts semi-regularly on $\Omega _{N}$. If $N$ is finite
then $k\left( N\right) =k\left( Z\right) k\left( N/Z\right) $.
\end{theorem}

\begin{corollary}
\label{Coro_NAbelianWithTransversal}Let $G:=M\circ _{Z}N$ and suppose that $%
Z $ has a normal transversal $Y$ in $N$. Then $G=M\times Y$. In particular,
if $N$ is an abelian group, and $Y$ is any transversal of $Z$ in $N$, then $%
G=M\times Y$.
\end{corollary}

Here we prove Corollary \ref{Coro_NAbelianWithTransversal} by showing
directly how the conditions of Definition \ref{Def_DirectProductOfSubsets}
are fulfilled. In Section \ref{Section_SetDirFacOfGps} we give a second
proof of Corollary \ref{Coro_NAbelianWithTransversal} which relies on its
connection with Theorem \ref{Th_XhasNormalTransversal}.

\begin{proof}[Proof of Corollary \protect\ref{Coro_NAbelianWithTransversal}]

Since $G=MN$, $N=ZY$ and $Z\leq M$ we get $G=MZY=MY$. If $%
m_{1}y_{1}=m_{2}y_{2}$ for $m_{1},m_{2}\in M$ and $y_{1},y_{2}\in Y$ then $%
y_{1}y_{2}^{-1}\in M$ whereby $y_{1}y_{2}^{-1}\in M\cap N=Z$ and hence $%
y_{1}=y_{2}$ and $m_{1}=m_{2}$. This proves the unique factorization
property required in Definition \ref{Def_DirectProductOfSubsets}, and hence $%
G=M\times Y$.
\end{proof}

In the second special case we consider, a single set-direct factorization of
an abelian group $Z$ induces a set-direct factorization of $G=M\circ _{Z}N$.
Our non-standard commutator notation is explained in the first paragraph of
Section \ref{Sect_NotationBackground}.

\begin{theorem}
\label{Th_DPWithCyclicCenterNeceSuff} Let $G=M\circ _{Z}N$ be a group, and $%
Z=X_{0}\times Y_{0}$ a set-direct decomposition of $Z$. Assume that $\left[
M,M\right] \cap Z\subseteq K_{Z}\left( X_{0}\right) $ and $\left[ N,N\right]
\cap Z\subseteq K_{Z}\left( Y_{0}\right) $. Then $G$ has a set-direct
factorization $G=X\times Y$ with $X\subseteq M$, $Y\subseteq N$, $X\cap
Z=X_{0}$ and $Y\cap Z=Y_{0}$.
\end{theorem}

\begin{corollary}
\label{Coro_PrimePowerWithNoFixedPoints}Let $G$ be a group with a
non-trivial central element $z$ of prime power order which is not a prime.
Suppose that $z$ is semi-regular. Then $G$ has a non-trivial set-direct
factorization. Furthermore, if $G$ is also perfect, then $G$ has a
non-trivial normalized (see Remark \ref{Rem_Normalization}) set-direct
factorization such that none of the factors is a group.
\end{corollary}

Concrete examples of non-trivial set-direct factorizations of the type
described by Corollary \ref{Coro_PrimePowerWithNoFixedPoints}, where $G$ is
a non-abelian finite quasi-simple group, are provided by the following
theorem whose proof rests on the work of Blau in \cite{Blau1994}.

\begin{theorem}
\label{Th_QuasiSimple}Let $G$ be a finite quasi-simple group. Then $G$ has a
non-trivial normalized set-direct decomposition if and only if $G/Z\left(
G\right) \cong PSL\left( 3,4\right) $ and $8$ divides $\left\vert Z\left(
G\right) \right\vert $. Moreover, $G$ has a non-trivial normalized
set-direct decomposition such that none of the two factors is a group if and
only if $G/Z\left( G\right) \cong PSL\left( 3,4\right) $ and $16$ divides $%
\left\vert Z\left( G\right) \right\vert $.
\end{theorem}

\section{Notation and Background Results\label{Sect_NotationBackground}}

Let $G$ be any group. For $x\in G$ we denote the conjugacy class of $x$ in $%
G $ by $x^{G}$. For any normal subset $S$ of $G$ let $\Omega _{S}$ denote
the set of all conjugacy classes of $G$ contained in $S$. We set $k\left(
G\right) :=\left\vert \Omega _{G}\right\vert $ in the case that $G$ is
finite. For any two subsets $A$ and $B$ of $G$, we denote by $\left[ A,B%
\right] $ the set of all commutators $\left[ a,b\right] :=a^{-1}b^{-1}ab$
where $a$ and $b$ vary over all elements of $A$ and $B$ respectively. Note
that for subgroups $A$ and $B$ our notation differs from the common practice
to denote by $\left[ A,B\right] $ the subgroup generated by all $\left[ a,b%
\right] $ with $a\in A$ and $b\in B$. If $A=\left\{ a\right\} $ we may write 
$\left[ a,B\right] $ for $\left[ \left\{ a\right\} ,B\right] $ and similarly 
$\left[ A,b\right] :=\left[ A,\left\{ b\right\} \right] $.

\subsection{The action of a central subgroup on conjugacy classes\label%
{SubSect_CentralMultiplicationAction}}

In this subsection we summarize basic properties of the multiplication
action of a central subgroup $Z\leq G$ on the conjugacy classes of $G$.

\begin{lemma}
\label{Lem_ZAction}Let $G$ be a group, $N\trianglelefteq G$, and $Z\leq N$ a
central subgroup of $G$.
\end{lemma}

\begin{enumerate}
\item $Z$ acts by multiplication on $\Omega _{N}$, namely, for any $z\in Z$
and $D\in \Omega _{N}$ we have $zD=Dz\in \Omega _{N}$. Denote this action by 
$\alpha $.

\item For any $D\in \Omega _{N}$ denote by $Z_{D}\leq Z$ the stabilizer of $%
D $ with respect to $\alpha $. Then $z\in $ $Z_{D}$ if and only if there
exists some $d\in D$ such that $dz\in D$.

\item For any $D\in \Omega _{N}$ denote by $O_{D}$ the orbit of $D$ under $%
\alpha $. Then for any $D_{1},D_{2}\in O_{D}$ we have $Z_{D_{1}}=Z_{D_{2}}$.

\item Let $n\in N$ and let $D=n^{G}$. Then $Z_{D}=\left[ n,G\right] \cap Z$.

\item Let $Y$ be a normal subset of $G$ contained in $N$. Let $n\in N$ and
let $D=n^{G}$. Set $Y_{n}:=n^{-1}Y\cap Z$. If $Y_{n}$ is non-empty then $%
Z_{D}\leq K\left( Y_{n}\right) $ (equivalently, $Y_{n}$ is a union of cosets
of $Z_{D}$ in $Z$).

\item Suppose that $G$ is a finite group. Then $k\left( G/Z\right) $ is
equal to the number of orbits of the multiplication action of $Z$ on $\Omega
_{G}$. It follows that $Z$ acts semi-regularly on $\Omega _{G}$ if and only
if $k\left( G\right) =k\left( Z\right) k\left( G/Z\right) $.
\end{enumerate}

\begin{proof}
\begin{enumerate}
\item Note that $N$ is the union of all elements of $\Omega _{N}$ and hence $%
\Omega _{N}\neq \emptyset $. Let $D\in \Omega _{N}$ and $z\in Z$. Since $%
Z\leq N$ we have $Dz\subseteq N$. We have to show that $Dz$ is also a
conjugacy class of $G$. Let $y\in Dz$ and $g\in G$. Then there exists $d\in
D $ such that $y=dz$. Using the fact that $Z$ is central, we have 
\begin{equation*}
y^{g}=\left( dz\right) ^{g}=d^{g}z\in Dz\text{.}
\end{equation*}%
Thus, $Dz$ is a normal subset of $G$. Now suppose that $d_{1},d_{2}\in D$.
Then there exists $g\in G$ such that $d_{2}=d_{1}^{g}$, and hence $%
d_{2}z=\left( d_{1}z\right) ^{g}$. Thus, any two elements in $Dz$ are
conjugate in $G$. This completes the proof that $Dz$ is also a conjugacy
class of $G$.

\item One direction is trivial. In the other direction let $d\in D$ be such
that $dz\in D$. Then $D\cap Dz\neq \emptyset $ and since both sets are
conjugacy classes this forces $Dz=D$.

\item Since $Z$ acts transitively on $O_{D}$, the stabilizers $Z_{D_{1}}$
and $Z_{D_{2}}$ are conjugate in $Z$, and since $Z$ is abelian, this implies 
$Z_{D_{1}}=Z_{D_{2}}$.

\item Let $z\in Z_{D}\leq Z$. Then $nz\in D$. This implies that there exists 
$x\in G$ such that $nz=x^{-1}nx$. Hence $z=n^{-1}x^{-1}nx=\left[ n,x\right]
\in \left[ n,G\right] $. Therefore $z\in \left[ n,G\right] \cap Z$.
Conversely, let $z\in \left[ n,G\right] \cap Z$. Then there exists $x\in G$
such that $z=\left[ n,x\right] =n^{-1}x^{-1}nx$. It follows that $%
nz=x^{-1}nx\in D$. By part 2, this implies $z\in Z_{D}$.

\item We have to show $Y_{n}Z_{D}=Y_{n}$. Let $\widetilde{Y}%
_{n}:=nY_{n}=Y\cap nZ$. Set $U:=Y\cap DZ$. Then, since $nZ\subseteq DZ$ we
have $\widetilde{Y}_{n}=U\cap nZ$. On the other hand, since $Y$ is normal in 
$G\,$, we get that $U$ is a union of classes in $O_{D}$, and hence, by part
3, $UZ_{D}=U$. This together with $\left( nZ\right) Z_{D}=nZ$ implies $%
\widetilde{Y}_{n}Z_{D}=\widetilde{Y}_{n}$. Finally, this implies $%
Y_{n}Z_{D}=n^{-1}\widetilde{Y}_{n}Z_{D}=n^{-1}\widetilde{Y}_{n}=Y_{n}$.

\item We first prove that $k\left( G/Z\right) $ is equal to the number of
orbits of the multiplication action of $Z$ on $\Omega _{G}$. Let $D\in
\Omega _{G}$. It is easy to check that the set $DZ$, viewed as a set of
cosets of $Z$ in $G$ is a conjugacy class of $G/Z$. By part 1, $DZ$ can also
be viewed as a disjoint union of conjugacy classes of $G$ which form one
orbit under the multiplication action of $Z$. Hence if $\widetilde{D}\in
\Omega _{G}$ then either $\widetilde{D}\subseteq DZ$ or $\widetilde{D}Z$ is
a distinct conjugacy class of $G/Z$. Moreover, every conjugacy class of $G/Z$
is of the form $DZ$ for some $D\in \Omega _{G}$. Therefore, $k\left(
G/Z\right) $ is equal to the number of orbits of the multiplication action
of $Z$ on $\Omega _{G}$. Now, the length of each orbit is at most $%
\left\vert Z\right\vert $, and $k\left( G\right) $ equals the sum of the
lengths of all of the $Z$-orbits. Since the action of $Z$ is semi-regular if
and only if all of the lengths equal $\left\vert Z\right\vert =k\left(
Z\right) $, we get that the action is semi-regular if and only if $k\left(
G\right) =k\left( Z\right) k\left( G/Z\right) $.
\end{enumerate}
\end{proof}

\begin{remark}
\label{Rem_NormalSetNotation}We make two notational remarks in relation to
Lemma \ref{Lem_ZAction}.

\begin{enumerate}
\item Let $G$ be a group and let $X$ be a set of conjugacy classes of $G$.
In some discussions (e.g., $X$ is an orbit of $Z$) it is convenient to abuse
notation and let $X$ stand also for the normal $G$-subset $%
\tbigcup\limits_{C\in X}C$. Conversely, if $X$ is a normal subset of $G$ we
may use the same letter $X$ to denote the set of all conjugacy classes of $G$
which are contained in $X$. We trust the reader to figure out the correct
interpretation from the context.

\item In view of the third claim of Lemma \ref{Lem_ZAction}, we also write $%
Z_{O_{C}}$ instead of $Z_{C}$.
\end{enumerate}
\end{remark}

\subsection{Central products and their conjugacy classes\label%
{Subsect_CentralProducts}}

The following theorem is at the basis of the construction of central
products of groups.

\begin{theorem}[\protect\cite{Gorenstein2}, Theorem 2.5.3]
Let $M,N,Z$ be groups with $Z\leq Z\left( M\right) $, and suppose that there
is an isomorphism $\theta $ of $Z$ into $Z\left( N\right) $. Then, if we
identify $Z$ with its image $\theta \left( Z\right) $, there exists a group $%
G$ of the form $G=MN$, with $Z=M\cap N\leq Z\left( G\right) $ such that $M$
centralizes $N$.
\end{theorem}

Any group $G$ with $M,N,Z$ as in the theorem is said to be the central
product of $M$ and $N$ (with respect to $Z$), and we will write $G=M\circ
_{Z}N$.

Next we consider the structure of the conjugacy classes of a central product 
$G=M\circ _{Z}N$, and the multiplication action of $Z$ on them (see Section %
\ref{SubSect_CentralMultiplicationAction}).

\begin{lemma}
\label{Lem_ConjClassesOfCentralProduct}Let $G$ be a group and $M$ and $N$
normal subgroups of $G$, such that $G=M\circ _{Z}N$, where $Z:=$ $M\cap N$.
Let $C$ be a conjugacy class of $G$. Then:

\begin{enumerate}
\item There exist a conjugacy class $C_{M}$ of $M$ and a conjugacy class $%
C_{N}$ of $N$ such that $C=C_{M}C_{N}$.

\item Let $\left\{ \left( C_{M}^{\left( i\right) },C_{N}^{\left( i\right)
}\right) \right\} _{i\in I}$ be the set of all the distinct pairs of
conjugacy classes $C_{M}^{\left( i\right) }\in \Omega _{M}$ and $%
C_{N}^{\left( i\right) }\in \Omega _{N}$ such that $C=C_{M}^{\left( i\right)
}C_{N}^{\left( i\right) }$ for all $i\in I$. Then each one of $%
O_{M}:=\left\{ C_{M}^{\left( i\right) }\right\} _{i\in I}$ and $%
O_{N}:=\left\{ C_{N}^{\left( i\right) }\right\} _{i\in I}$ is a single orbit
of the multiplication action of $Z$ on $\Omega _{M}$ and $\Omega _{N}$
respectively.

\item $O_{C}=O_{M}O_{N}$ and $Z_{O_{C}}=Z_{O_{M}}Z_{O_{N}}$.

\item The equality $O_{C}=O_{M}O_{N}$ from part 3 defines a bijection $%
O\left( \Omega _{G}\right) \rightarrow O\left( \Omega _{M}\right) \times
O\left( \Omega _{N}\right) $.
\end{enumerate}
\end{lemma}

\begin{proof}
1. Let $g\in C$. Then, since $G=MN$, there exist $m\in M$ and $n\in N$ such
that $g=mn$. Hence:%
\begin{equation*}
C=g^{G}=\left\{ \left( mn\right) ^{xy}|x\in M,y\in N\right\} =\left\{
m^{x}n^{y}|x\in M,y\in N\right\} =m^{M}n^{N}\text{,}
\end{equation*}%
where the third equality relies on the fact that $M$ and $N$ centralize each
other. Now we can take $C_{M}:=m^{M}$ and $C_{N}:=n^{N}$.

2. Let $C_{M}$ and $C_{N}$ be as in the first part. Note that for any $z\in
Z $, $C=\left( z^{-1}C_{M}\right) \left( zC_{N}\right) $. On the other hand,
suppose that $C=A_{1}B_{1}$ where $A_{1}\in \Omega _{M}$ and $B_{1}\in
\Omega _{N}$. Then there exist $m_{1}\in A_{1}$ and $n_{1}\in B_{1}$ such
that $g=mn=m_{1}n_{1}$. Hence $z:=m_{1}^{-1}m=n_{1}n^{-1}\in Z$ and $%
m_{1}=z^{-1}m$ while $n_{1}=zn$. This implies $A_{1}=z^{-1}C_{M}$ and $%
B_{1}=zC_{N}$. We have proved: 
\begin{equation*}
(\ast )~~\left\{ \left( C_{M}^{\left( i\right) },C_{N}^{\left( i\right)
}\right) \right\} _{i\in I}=\left\{ \left( z^{-1}C_{M},zC_{N}\right) |z\in
Z\right\} \text{,}
\end{equation*}%
which implies the claim.

3. Let $C_{M}$ and $C_{N}$ be as in the first part. We have $O_{M}=\left\{
z_{1}C_{M}|z_{1}\in Z\right\} $ and $O_{N}=\left\{ z_{2}C_{N}|z_{2}\in
Z\right\} $. This gives 
\begin{eqnarray*}
O_{M}O_{N} &=&\left\{ z_{1}C_{M}z_{2}C_{N}|z_{1},z_{2}\in Z\right\} =\left\{
z_{1}z_{2}C_{M}C_{N}|z_{1},z_{2}\in Z\right\} \\
&=&\left\{ zC|z\in Z\right\} =O_{C}\text{.}
\end{eqnarray*}%
For the second claim, let $z\in Z_{C}$. Then%
\begin{equation*}
C_{M}C_{N}=C=zC=\left( zC_{M}\right) C_{N}\text{.}
\end{equation*}%
This implies that the pairs $\left( zC_{M},C_{N}\right) ,\left(
C_{M},C_{N}\right) \in \Omega _{M}\times \Omega _{N}$ yield the same
conjugacy class $C$. By $(\ast )$ there exist $z^{\prime }\in Z$ such that $%
z^{\prime -1}zC_{M}=C_{M}$ and $z^{\prime }C_{N}=C_{N}$. Hence $z^{\prime
}\in $ $Z_{C_{N}}$ and $z^{\prime -1}z\in Z_{C_{M}}$, and since $z=z^{\prime
}\left( z^{\prime -1}z\right) $ we get $z\in Z_{C_{M}}Z_{C_{N}}$. This
proves that $Z_{C}\leq Z_{C_{M}}Z_{C_{N}}$. On the other hand, for any $%
z_{1}\in Z_{C_{M}}$ and $z_{2}\in Z_{C_{N}}$ we get 
\begin{equation*}
z_{1}z_{2}C=z_{1}C_{M}z_{2}C_{N}=C_{M}C_{N}=C\text{.}
\end{equation*}%
This proves that $Z_{C_{M}}Z_{C_{N}}\leq Z_{C}$ and altogether we get $%
Z_{C}=Z_{C_{M}}Z_{C_{N}}$.

4. By part 3, $O_{C}=O_{M}O_{N}$ with $O_{M}\in O\left( \Omega _{M}\right) $
and $O_{N}\in O\left( \Omega _{N}\right) $. Moreover, the proofs of parts 2
and 3 show that $O_{C}$ uniquely determines $\left( O_{M},O_{N}\right) $,
and that each pair $\left( O_{M},O_{N}\right) $ uniquely determines an orbit 
$O_{C}$ of $G$-conjugacy classes under the multiplication action of $Z$.
\end{proof}

The following type of subset plays an important role in the analysis of the
action of a central subgroup on the conjugacy classes of central products.

\begin{definition}
\label{Def_Z_[K]}Let $G$ be a group, $Z$ a central subgroup of $G$ and $%
K\trianglelefteq G$. We set $Z_{\left[ K\right] }:=\left[ K,K\right] \cap Z$.
\end{definition}

\begin{lemma}
\label{Lem_Z_M_and_Z_N}Let $G$ be a group and $M$ and $N$ normal subgroups
of $G$, such that $G=M\circ _{Z}N$, where $Z:=$ $M\cap N$.

\begin{enumerate}
\item[(a)] 
\begin{equation*}
\left[ M,M\right] \cap \left[ N,N\right] =Z_{\left[ M\right] }\cap Z_{\left[
N\right] }\text{.}
\end{equation*}

\item[(b)] Let $I$ and $J$ be indexing sets such that $O\left( \Omega
_{M}\right) :=\left\{ X_{i}\right\} _{i\in I}$ and $O\left( \Omega
_{N}\right) =\left\{ Y_{j}\right\} _{j\in J}$. Then:%
\begin{equation*}
Z_{\left[ M\right] }=\tbigcup\limits_{i\in I}Z_{X_{i}}\text{ , }Z_{\left[ N%
\right] }=\tbigcup\limits_{j\in J}Z_{Y_{j}}\text{ and }Z_{\left[ G\right]
}=\tbigcup\limits_{i\in I}\tbigcup\limits_{j\in J}Z_{X_{i}Y_{j}}=Z_{\left[ M%
\right] }Z_{\left[ N\right] }\text{.}
\end{equation*}
\end{enumerate}
\end{lemma}

\begin{proof}
(a) Since $G=M\circ _{Z}N$, we have $\left[ M,M\right] \cap \left[ N,N\right]
\subseteq Z$ and hence, by Definition \ref{Def_Z_[K]}, 
\begin{equation*}
\left[ M,M\right] \cap \left[ N,N\right] =\left[ M,M\right] \cap \left[ N,N%
\right] \cap Z=Z_{\left[ M\right] }\cap Z_{\left[ N\right] }\text{.}
\end{equation*}

(b) We prove $Z_{\left[ N\right] }=\tbigcup\limits_{j\in J}Z_{Y_{j}}$. First
note that since $M$ centralizes $N$, we have, for any $n\in N$, $\left[ n,G%
\right] =\left[ n,N\right] $. Let $j\in J$. By Lemma \ref{Lem_ZAction} parts
(3) and (4) and Remark \ref{Rem_NormalSetNotation} part 2, if $n\in N$
belongs to a conjugacy class in the orbit $Y_{j}$ then $Z_{Y_{j}}=\left[ n,G%
\right] \cap Z=\left[ n,N\right] \cap Z$. Therefore 
\begin{equation*}
\tbigcup\limits_{j\in J}Z_{Y_{j}}=\tbigcup\limits_{n\in N}\left( \left[ n,N%
\right] \cap Z\right) =Z\cap \tbigcup\limits_{n\in N}\left[ n,N\right] =Z_{%
\left[ N\right] }\text{.}
\end{equation*}

The proof that $Z_{\left[ M\right] }=\tbigcup\limits_{i\in I}Z_{X_{i}}$ and $%
Z_{\left[ G\right] }=\tbigcup\limits_{i\in I}\tbigcup\limits_{j\in
J}Z_{X_{i}Y_{j}}$ is similar, and $Z_{\left[ G\right] }=Z_{\left[ M\right]
}Z_{\left[ N\right] }$ follows from Lemma \ref%
{Lem_ConjClassesOfCentralProduct} part (3).
\end{proof}

\subsection{Basic properties of a set-direct product}

The following lemma states several equivalent conditions for the directness
of a setwise product. Although the group is not assumed to be abelian the
proof is essentially the same as that of \cite[Lemma 2.2]{SzaboSandsBook2009}%
, and is therefore omitted.

\begin{lemma}
\label{Lem_EquivDirectnessConditions}Let $G$ be a group, and let $X$ and $Y$
be two non-empty normal subsets of $G$. The following conditions are
equivalent.

\begin{enumerate}
\item[(a)] $XY=X\times Y$.

\item[(b)] $XX^{-1}\cap YY^{-1}=\left\{ 1_{G}\right\} $.

\item[(c)] $\left\{ Xy|y\in Y\right\} $ or $\left\{ xY|x\in X\right\} $ is a
partition of $XY$.
\end{enumerate}

Furthermore, if $X$ and $Y$ are finite sets then $XY$ is direct if and only
if $\left\vert XY\right\vert =\left\vert X\right\vert \cdot \left\vert
Y\right\vert $.
\end{lemma}

\begin{remark}
\label{Rem_XIntersectY}Let $G$ be a group, and let $X$ and $Y$ be two
non-empty normal subsets of $G$. Then $XY=X\times Y$ implies $\left\vert
X\cap Y\right\vert \leq 1$. To see this observe that if $a,b\in X\cap Y$,
then $ab=ba$ are two factorizations in $X\times Y$, and hence, by uniqueness
of factorization, $a=b$.
\end{remark}

\begin{lemma}
\label{Lem_TrivialCentersImplyIdentityInside}Let $G$ be a group, and let $%
G=X\times Y$ be a set-direct decomposition of $G$. Then:

\begin{enumerate}
\item[(a)] If $C$ is a conjugacy class of $G$ contained in $X$, and $%
\left\vert C\right\vert >1$ then $Y\cap C=Y\cap C^{-1}=\emptyset $.

\item[(b)] There exists a central element $z$ of $G$ such that $z\in X$ and $%
z^{-1}\in Y$.

\item[(c)] If $z$ is any central element of $G$ then $G=\left( zX\right)
\times Y=X\times \left( zY\right) $.
\end{enumerate}
\end{lemma}

\begin{proof}

(a) Since $\left\vert C\right\vert >1$ then $C^{-1}C=CC^{-1}$ must contain
non-trivial elements, and hence, if $Y$ contains $C$ or $C^{-1}$ we get a
contradiction with Lemma \ref{Lem_EquivDirectnessConditions}(b).

(b) By part (a), if $C$ is a conjugacy class, $C\subseteq X$ and $%
C^{-1}\subseteq Y$, then $C$ must consist of a single central element. On
the other hand, since $1_{G}\in G$, we must have at least one class $C$ such
that $C\subseteq X$ and $C^{-1}\subseteq Y$.

(c) Since $z$ is central, the normality of $X$ implies the normality of $zX$
and $\left( zX\right) ^{-1}\left( zX\right) =X^{-1}X$. Similar claims hold
when $X$ is replaced by $Y$. Now apply Lemma \ref%
{Lem_EquivDirectnessConditions}.
\end{proof}

\begin{remark}
\label{Rem_Normalization}In view of Lemma \ref%
{Lem_TrivialCentersImplyIdentityInside}(c), $Z\left( G\right) \times Z\left(
G\right) $ acts on the set of all direct factorizations of $G$ via $X\times
Y\longmapsto \left( z_{1}X\right) \times \left( z_{2}Y\right) $ where $%
z_{1},z_{2}\in Z\left( G\right) $ and by Lemma \ref%
{Lem_TrivialCentersImplyIdentityInside}(b), each orbit of this action
contains at least one normalized factorization. Note that if $X\subseteq G$
is normalized and $X$ is not a subgroup of $G$ then $gX$ is not a subgroup
of $G$ for any $g\in G$. For suppose by contradiction that $gX$ is a
subgroup for some $g\in G$. Then, since $1_{G}\in X$ we get that $g\in gX$
and hence $g^{-1}\in gX$, implying $gX=g^{-1}\left( gX\right) =X$, whereby $%
X $ is a subgroup - a contradiction.
\end{remark}

\begin{lemma}
\label{Lem_AssociationOfDirect}Let $G$ be a group, and let $A,B,C$ be three
non-empty normal subsets of $G$. Suppose that the products $AB$ and $\left(
AB\right) C$ are direct. Then the products $BC$ and $A\left( BC\right) $ are
also direct.
\end{lemma}

\begin{proof}
First we show that $A\left( BC\right) $ is direct. For this we have to show
that if $a_{1},a_{2}\in A$, $b_{1},b_{2}\in B$ and $c_{1},c_{2}\in C$ and 
\begin{equation*}
a_{1}\left( b_{1}c_{1}\right) =a_{2}\left( b_{2}c_{2}\right) \text{,}
\end{equation*}%
then $a_{1}=a_{2}$, and $b_{1}c_{1}=b_{2}c_{2}$. By associativity, $\left(
a_{1}b_{1}\right) c_{1}=\left( a_{2}b_{2}\right) c_{2}$. Since $\left(
AB\right) C$ is direct, we get $c_{1}=c_{2}$ and $a_{1}b_{1}=a_{2}b_{2}$.
Since $AB$ is direct we further get $a_{1}=a_{2}$ and\ $b_{1}=b_{2}$. It
follows that $b_{1}c_{1}=b_{2}c_{2}$. Now we show that $BC$ is direct. Let $%
b_{1},b_{2}\in B$, $c_{1},c_{2}\in C$ and suppose that $%
b_{1}c_{1}=b_{2}c_{2} $. Multiplying on the left by some $a\in A$ (recall
that by assumption $A\neq \emptyset $) \ and using associativity, we have 
\begin{equation*}
\left( ab_{1}\right) c_{1}=\left( ab_{2}\right) c_{2}\text{.}
\end{equation*}%
Since $\left( AB\right) C$ is direct, this gives $c_{1}=c_{2}$ and $%
ab_{1}=ab_{2}$, which yields $b_{1}=b_{2}$.
\end{proof}

\section{Set-direct factorizations of groups\label{Section_SetDirFacOfGps}}

In this section we prove the various conditions stated in the Introduction
for set-direct decompositions of groups. We begin by showing that if the
product of two subsets of a group is direct then they must centralize each
other.

\begin{theorem}
\label{Th_DirectImpliesCentralizing}Let $G$ be a group and let $X$ and $Y$
be two normal subsets of $G$. If $XY=X\times Y$ then $\left[ X,Y\right]
=\left\{ 1_{G}\right\} $.
\end{theorem}

\begin{proof}
Let $X$ and $Y$ be two normal subsets of $G$, and assume that $XY=X\times Y$%
. Let $x\in X$, $y\in Y$ and $t:=$ $xy$. It is clear that $C_{G}\left(
x\right) \cap C_{G}\left( y\right) \leq C_{G}\left( t\right) $. We prove
that $C_{G}\left( t\right) \leq C_{G}\left( x\right) \cap C_{G}\left(
y\right) $. Let $h\in C_{G}\left( t\right) $. Then $xy=t=t^{h}=x^{h}y^{h}$.
Since $x^{h}\in X$ and $y^{h}\in Y$, we have, by uniqueness of
factorization, that $x^{h}=x$ and $y^{h}=y$, implying $h\in C_{G}\left(
x\right) \cap C_{G}\left( y\right) $, and $C_{G}\left( t\right) =C_{G}\left(
x\right) \cap C_{G}\left( y\right) $. In particular, $t\in C_{G}\left(
x\right) \cap C_{G}\left( y\right) $. Hence $t$ commutes with both $x$ and $%
y $. But $t=$ $xy$ implies $y=x^{-1}t$. Using the fact that $x$ commutes
with both $x^{-1}$ and $t$ we get that $x$ and $y$ commute.
\end{proof}

Using Theorem \ref{Th_DirectImpliesCentralizing}, we obtain Theorem \ref%
{Th_MainNecessaryAndSufficient} as a consequence of an apparently more
general statement.

Let $G$ be a group and let $X$ and $Y$ be subsets of $G$. We write $%
G=X\times _{c}Y$ if every element $g\in G$ can be uniquely expressed as $%
g=xy $ where $x\in X$ and $y\in Y$, and if $X$ centralizes $Y$. Clearly, for
an abelian $G$, $G=X\times _{c}Y$ and $G=X\times Y$ is the same, and, in
general, $G=X\times Y$ implies $G=X\times _{c}Y$.

\begin{theorem}
\label{Th_G=Xx_cY} Let $G$ be a group and let $X$ and $Y$ be subsets of $G$.
Set $M:=\left\langle X\right\rangle $, $N:=\left\langle Y\right\rangle $ and 
$Z:=M\cap N$. For every $m\in M$ and $n\in N$ set $X_{m}:=\left(
m^{-1}X\right) \cap Z$ and $Y_{n}:=\left( n^{-1}Y\right) \cap Z$. Then $%
G=X\times _{c}Y$ if and only if

\begin{enumerate}
\item[(a)] $G=M\circ _{Z}N$; and

\item[(b)] $Z=X_{m}\times Y_{n}$ for every $m\in M$ and $n\in N$.
\end{enumerate}
\end{theorem}

\begin{proof}
Assume that $G=X\times _{c}Y$. Then $M$\ is centralized by $Y$\ and hence $%
M\trianglelefteq G$. Similarly, $N\trianglelefteq G$. Furthermore, $M$
centralizes $N$. Thus (a) follows.

Now we prove (b). From $G=X\times _{c}Y$ and by (a) we have $G=XY=M\circ
_{Z}N$. Furthermore, $G=X\times _{c}Y$ and $N=\left\langle Y\right\rangle $
imply that $n$ centralizes $X$ for all $n\in N$. Similarly $m$ centralizes $%
Y $ for all $m\in M$. Hence, for all $m\in M$ and $n\in N$ we have 
\begin{equation*}
G=m^{-1}n^{-1}XY=m^{-1}Xn^{-1}Y\text{,}
\end{equation*}%
and this implies $G=\left( m^{-1}X\right) \times _{c}\left( n^{-1}Y\right) $
for all $m\in M$ and $n\in N$ (note that uniqueness of factorization follows
from that of $G=X\times _{c}Y$). In particular, for all $z\in Z$ there exist
unique $x\in X$ and $y\in Y$ such that $z=\left( m^{-1}x\right) \left(
n^{-1}y\right) $. This gives $m^{-1}x=y^{-1}nz$. As the l.h.s. is in $M$ and
the r.h.s. is in $N$ we get $m^{-1}x\in \left( m^{-1}X\right) \cap Z=X_{m}$
and similarly, $n^{-1}y\in Y_{n}$. This proves $Z=X_{m}\times Y_{n}$.

Conversely, assume that conditions (a) and (b) hold true. By (a) we get that 
$X\subseteq M$ centralizes $Y\subseteq N$. Let $g\in G$. By (a) there exist $%
m\in M$ and $n\in N$ such that $g=mn$. Using (b) and the fact that $X_{m}$
and $Y_{n}$ are central we get:%
\begin{equation*}
gZ=\left( mn\right) X_{m}Y_{n}=\left( mX_{m}\right) \left( nY_{n}\right)
=\left( X\cap mZ\right) \left( Y\cap nZ\right) \subseteq XY\text{.}
\end{equation*}%
Since $g\in gZ$ this implies the existence of $x\in X$ and $y\in Y$ such
that $g=xy$. It remains to prove uniqueness of representation. Suppose that
we also have $x_{1}\in X$ and $y_{1}\in Y$ such that $g=x_{1}y_{1}$. Then $%
xy=x_{1}y_{1}$ implying $x_{1}^{-1}xy=y_{1}$. Since $x_{1}^{-1}x\in M$ and $%
y\in N$, this implies $yx_{1}^{-1}x=y_{1}$ from which $%
x_{1}^{-1}x=y^{-1}y_{1}$ follows. Since the l.h.s. of the last equality is
in $M$ and the r.h.s. is in $N$ we get $x_{1}^{-1}x=y^{-1}y_{1}\in Z$. It
now follows that $x_{1}^{-1}x\in X_{x_{1}}$ and $y^{-1}y_{1}\in Y_{y}$,
hence $x_{1}^{-1}x=y^{-1}y_{1}\in X_{x_{1}}\cap Y_{y}$. Now, since $x_{1}\in
X$ we have $1_{G}\in X_{x_{1}}$ and similarly $1_{G}\in Y_{y}$. Therefore $%
1_{G}\in X_{x_{1}}\cap Y_{y}$. By (b), the product $X_{x_{1}}Y_{y}$ is
direct, therefore, by Remark \ref{Rem_XIntersectY}, $X_{x_{1}}\cap
Y_{y}=\left\{ 1_{G}\right\} $. This implies $x_{1}=x$ and $y_{1}=y$.
\end{proof}

\begin{proof}[Proof of Theorem \protect\ref{Th_MainNecessaryAndSufficient}]
Combine Theorems \ref{Th_DirectImpliesCentralizing} and \ref{Th_G=Xx_cY}.
\end{proof}

\begin{proof}[Proof of Theorem \protect\ref{Th_XxYAndFactorizationSystems}]
Since condition (a) of Theorem \ref{Th_MainNecessaryAndSufficient} holds by
assumption, we have that $G=X\times Y$ if and only if $Z=X_{m}\times Y_{n}$
for every $m\in M$ and $n\in N$. Hence, by Definition \ref%
{Def_FactorizationSystem}, it remains to show that for any $m\in M$ and $%
n\in N$ we have $M_{m}\leq K\left( X_{m}\right) $ and $N_{n}\leq K\left(
Y_{n}\right) $. This is immediate from Lemma \ref{Lem_ZAction} (5).
\end{proof}

\begin{remark}
\label{Rem_redundancy}There is a considerable redundancy in the description
of the factorization system $\left( \left\{ X_{m}\right\} _{m\in M},\left\{
Y_{n}\right\} _{n\in N}\right) $ of $Z$ used in Theorem \ref%
{Th_XxYAndFactorizationSystems}. First of all, since $Z$ is central we get
that $X_{m_{1}}=X_{m_{2}}$ for any two conjugate $m_{1},m_{2}\in M$, or
equivalently, $X_{m}$ depends only on the conjugacy class of $m$ in $G$.
Furthermore, suppose that $C_{1}$ and $C_{2}$ are two conjugacy classes of $%
M $ which belong to the same $Z$-orbit. Let $m_{i}\in C_{i}$ for $i=1,2$ and
let $c_{12}\in Z$ be such that $C_{2}=c_{12}C_{1}$. Then by Lemma \ref%
{Lem_ZAction} (3) we have $M_{m_{1}}=M_{m_{2}}$. Furthermore, we have $%
X_{m_{2}}=c_{12}^{-1}X_{m_{1}}$. Since similar claims hold true for the $%
Y_{n}$ it follows that one can replace the indexing sets $M$ and $N$ by,
respectively, the set of orbits of the multiplication action of $Z$ on $%
\Omega _{M}$ and on $\Omega _{N}$, and replace $\left( \left\{ X_{m}\right\}
_{m\in M},\left\{ Y_{n}\right\} _{n\in N}\right) $ by a "trimmed"
factorization system.
\end{remark}

\begin{theorem}
\label{Th_FacSysProperties}Let $Z$ be an abelian group and let $\mathcal{M}%
=\left\{ M_{i}\right\} _{i\in I}$ and $\mathcal{N}=\left\{ N_{j}\right\}
_{j\in J}$ be two multisets of subgroups of $Z$. Let $\left( \mathcal{A},%
\mathcal{B}\right) $ be an $\mathcal{MN}$-direct factorization system of $Z$
where $\mathcal{A}=\left\{ A_{i}\right\} _{i\in I}$ and $\mathcal{B}=\left\{
B_{j}\right\} _{j\in J}$. Then:

\begin{enumerate}
\item[(a)] For any $i\in I$ and $j\in J$, any two distinct elements $%
a_{1},a_{2}\in A_{i}$ belong to distinct cosets of $N_{j}$ in $Z$ and any
two distinct elements $b_{1},b_{2}\in B_{j}$ belong to distinct cosets of $%
M_{i}$ in $Z$.

\item[(b)] Suppose that $Z$ is finite. Then $\left( \mathcal{A},\mathcal{B}%
\right) $ has to satisfy the following arithmetical conditions:%
\begin{equation}
\left\vert Z\right\vert =\left\vert A_{i}\right\vert \left\vert
B_{j}\right\vert \text{, }\forall 1\leq i\leq r\text{, }\forall 1\leq j\leq s%
\text{,}  \label{Eq_|Z|=|A_i||B_j|}
\end{equation}%
which imply 
\begin{equation}
\left\vert A_{1}\right\vert =\cdots =\left\vert A_{r}\right\vert \text{, }%
\left\vert B_{1}\right\vert =\cdots =\left\vert B_{s}\right\vert \text{,}
\label{Eq_lA_i|AreAllEqual}
\end{equation}%
and also 
\begin{equation}
\func{lcm}\left( \left\vert M_{1}\right\vert ,...,\left\vert
M_{r}\right\vert \right) \text{ divides }\left\vert A_{1}\right\vert \text{
and }\func{lcm}\left( \left\vert N_{1}\right\vert ,...,\left\vert
N_{s}\right\vert \right) \text{ divides }\left\vert B_{1}\right\vert \text{.}
\label{Eq_LcmCondition}
\end{equation}
\end{enumerate}
\end{theorem}

\begin{proof}

\begin{enumerate}
\item[(a)] We prove the first part of the claim. The proof of the second
part is essentially the same. Suppose, by contradiction, that there exist $%
i\in I$ and $j\in J$, and two elements $a_{1}\neq a_{2}\in A_{i}$ which
belong to the same coset of $N_{j}$ in $Z$. Then $a_{1}=n_{1}x$ and $%
a_{2}=n_{2}x$, where $n_{1}\neq n_{2}$ are elements of $N_{j}$ and $x\in Z$.
Since $B_{j}$ is a non-empty union of cosets of $N_{j}$, there exists some $%
y\in Z$, such that $N_{j}y\subseteq B_{j}$. Hence $b_{1}:=n_{1}y$ and $%
b_{2}:=n_{2}y$ belong to $B_{j}$ and $a_{1}b_{2}=a_{2}b_{1}$ are two
distinct factorizations in $A_{i}\times B_{j}$ of the same element - a
contradiction.

\item[(b)] Equation \ref{Eq_|Z|=|A_i||B_j|} is immediate from Lemma \ref%
{Lem_EquivDirectnessConditions}. For Equation \ref{Eq_LcmCondition} note
that $\left\vert K\left( A\right) \right\vert $ divides $\left\vert
A\right\vert $ for any finite, non-empty $A$. Hence, by Definition \ref%
{Def_FactorizationSystem}, $\left\vert M_{i}\right\vert $ divides $%
\left\vert A_{i}\right\vert $ for all $1\leq i\leq r$. Similarly, $%
\left\vert N_{j}\right\vert $ divides $\left\vert B_{j}\right\vert $ for all 
$1\leq j\leq s$. Now Equation \ref{Eq_LcmCondition} follows from Equation %
\ref{Eq_lA_i|AreAllEqual}.
\end{enumerate}
\end{proof}

\begin{proof}[Proof of Theorem \protect\ref{Th_ConstructingXxYFromFacSys}]
Fix arbitrary $i\in I$ and $j\in J$. Since $M_{i}\leq K\left( A_{i}\right) $
we get that $A_{i}$ is a union of cosets of $M_{i}$, and similarly, $B_{j}$
is a union of cosets of $N_{j}$. Therefore we can write $A_{i}=\tbigcup%
\limits_{s}a_{is}M_{i}$ where the $a_{is}\in Z$ represent distinct cosets of 
$M_{i}$ and similarly $B_{j}=\tbigcup\limits_{t}b_{jt}N_{j}$, with $b_{jt}$
representing distinct cosets of $N_{j}$. Hence all the $M$-conjugacy classes 
$a_{is}C_{i}$, which are contained in $A_{i}C_{i}$, are distinct for
distinct values of $s$, and similarly, the $N$-conjugacy classes $%
b_{jt}D_{j} $ which are contained in $B_{j}D_{j}$ are distinct for distinct
values of $t$. We claim that $\tbigcup\limits_{s,t}\left\{
a_{is}b_{jt}\right\} $ is a transversal of $Z_{ij}$ in $Z$, where $ij$ is
the unique $Z$-orbit formed by the product of $i$ and $j$ (see Lemma \ref%
{Lem_ConjClassesOfCentralProduct} (3)). In the first place, using $%
Z_{ij}=M_{i}N_{j}$ (Lemma \ref{Lem_ConjClassesOfCentralProduct} (3)), we get%
\begin{equation*}
\left( \tbigcup\limits_{s,t}\left\{ a_{is}b_{jt}\right\} \right)
Z_{ij}=\left( \tbigcup\limits_{s}a_{is}M_{i}\right) \left(
\tbigcup\limits_{t}b_{jt}N_{j}\right) =A_{i}B_{j}=Z\text{.}
\end{equation*}%
Next, suppose that $a_{is}b_{jt}Z_{ij}=a_{is^{\prime }}b_{jt^{\prime
}}Z_{ij} $. This yields $a_{is}b_{jt}\in \left( a_{is^{\prime }}M_{i}\right)
\left( b_{jt^{\prime }}N_{j}\right) $ which implies the existence of $%
z_{1}\in M_{i} $ and $z_{2}\in N_{j}$ such that $a_{is}b_{jt}=\left(
a_{is^{\prime }}z_{1}\right) \left( b_{jt^{\prime }}z_{2}\right) $. Observe
that $a_{is^{\prime }}z_{1}\in A_{i}$ and $b_{jt^{\prime }}z_{2}\in B_{j}$.
By uniqueness of factorization in $A_{i}\times B_{j}$ we get $%
a_{is}=a_{is^{\prime }}z_{1}$ and $b_{jt}=b_{jt^{\prime }}z_{2}$. If $s\neq
s^{\prime }$, we have, by our choice, that $a_{is}$ and $a_{is^{\prime }}$
belong to distinct $M_{i}$ cosets in contradiction to $a_{is}=a_{is^{\prime
}}z_{1}$. Hence $s=s^{\prime }$ and similarly $t=t^{\prime }$. Now it
follows that $a_{is}b_{jt}C_{i}D_{j}$ are distinct conjugacy classes for
distinct pairs $\left( s,t\right) $, and $\tbigcup\limits_{s,t}\left\{
a_{is}b_{jt}\right\} C_{i}D_{j}=ij$. This shows that $G=X\times Y$.
\end{proof}

The next theorem is needed for the proof of Theorem \ref%
{Th_XhasNormalTransversal}.

\begin{theorem}
\label{Th_ZHasNormalTInN}Let $N$ be a group and $Z$ a central subgroup of $N$%
. Then the following conditions are equivalent:

\begin{enumerate}
\item[(1)] $Z$ has a normal transversal $Y$ in $N$.

\item[(2)] $Z$ acts semi-regularly on $\Omega _{N}$.
\end{enumerate}

Furthermore, if $N$ is finite then each one of (1) and (2) is equivalent to $%
k\left( N\right) =k\left( Z\right) k\left( N/Z\right) $.
\end{theorem}

\begin{proof}
Suppose that $Z$ has a normal transversal $Y$ in $N$. Then every element of $%
Y$ meets every coset of $Z$ in $N$ in precisely one element which is
equivalent to $\left\vert Y_{n}\right\vert =1$ for every $n\in N$. Let $D\in
\Omega _{N}$. By Lemma \ref{Lem_ZAction} (5) we have\ $Z_{D}\leq K\left(
Y_{n}\right) $, where $n\in D$. Since $\left\vert Y_{n}\right\vert =1$, we
get that $K\left( Y_{n}\right) =\left\{ 1_{Z}\right\} $ and therefore $%
Z_{D}=\left\{ 1_{Z}\right\} $ for all $D\in \Omega _{N}$ which is the claim
of (2).

Conversely, assume (2). Let $J$ denote the set of all distinct orbits of the
multiplication action of $Z$ on $\Omega _{N}$. For each $j\in J$ let $D_{j}$
be a conjugacy class belonging to the orbit $j$. We claim that $%
T:=\tbigcup\limits_{j\in J}D_{j}$ is a normal transversal of $Z$ in $N$. The
normality of $T$ in $N$ is clear as it is a union of conjugacy classes.
Since $Z$ acts transitively by multiplication on each orbit we have $%
ZD_{j}=j $ and hence $ZT=N$. Now suppose that $t_{1},t_{2}\in T$ satisfy $%
Zt_{1}=Zt_{2}$. Then there exist $z\in Z$ such that $t_{2}=zt_{1}$. If $%
z\neq 1$ then, by the semi-regularity assumption, $t_{1}$ and $t_{2}$ belong
to two distinct conjugacy classes of $N$, say $D_{1}$ and $D_{2}$
respectively, but, on the other hand, $D_{1}$ and $D_{2}$ belong to the same 
$Z$-multiplication orbit. This contradicts the construction of $T$. Hence $%
z=1$ and $t_{2}=t_{1}$.

Finally, the last claim of the theorem follows from Lemma \ref{Lem_ZAction}
(6).
\end{proof}

\begin{proof}[Proof of Theorem \protect\ref{Th_XhasNormalTransversal}]
By assumption, $M:=\left\langle X\right\rangle =X$. Suppose that $G=X\times
Y $. Then $G=X\circ _{Z}N$ follows from Theorem \ref%
{Th_MainNecessaryAndSufficient}(a). By Theorem \ref%
{Th_MainNecessaryAndSufficient}(b) we have $Z=X_{m}\times Y_{n}$ for every $%
m\in M$ and $n\in N$. Since $M=X$ we have $Z\subseteq X$ and $mZ\subseteq X$
for every $m\in M$ and hence $X_{m}=Z$. By uniqueness of representation in
the set-direct product $Z=X_{m}\times Y_{n}$, this forces $\left\vert
Y_{n}\right\vert =1$ for all $n\in N$.

Conversely, assume condition (2). Let $m\in M$ and $n\in N$. Then $X_{m}=Z$
and $\left\vert Y_{n}\right\vert =1$ implies $Z=X_{m}\times Y_{n}$. Now (1)
follows by Theorem \ref{Th_MainNecessaryAndSufficient}.

The condition that for any $n\in N$ we have $\left\vert Y_{n}\right\vert =1$%
, is equivalent to the statement that $Y$ intersects every coset of $Z$ in $%
N $ in precisely one element, which is equivalent to $Y$ being a transversal
of $Z$ in $N$. The remaining claims follow from Theorem \ref%
{Th_ZHasNormalTInN}.
\end{proof}

\begin{proof}[Second proof of Corollary \protect\ref%
{Coro_NAbelianWithTransversal}]

Let $Y$ be a normal transversal of $Z$ in $N$. Set $N_{1}=\left\langle
Y\right\rangle $. Clearly $N_{1}\leq N$ is a normal subgroup of $G$ and $%
Z_{1}:=M\cap N_{1}$ is a subgroup of $Z$. We have $G=MN=MZY=MY$ and hence $%
G=MN_{1}$ implying $G=M\circ _{Z_{1}}N_{1}$. Now, since $Y$ is a normal
transversal of $Z$ in $N$ it follows that $\left\vert Y_{n}\right\vert =1$
for all $n\in N$ (see last paragraph of the proof of Theorem \ref%
{Th_XhasNormalTransversal}) and hence for all $n\in N_{1}$. Furthermore, $%
Z_{1}\leq M$ and hence, setting $X:=M$, we get $X_{m}=\left( m^{-1}X\right)
\cap Z_{1}=Z_{1}$ for all $m\in M$. Thus, condition (2) of Theorem \ref%
{Th_XhasNormalTransversal} is satisfied when substituting $N_{1}$ for $N$
and $Z_{1}$ for $Z$. Consequently $G=M\times Y$.
\end{proof}

In order to prove Theorem \ref{Th_DPWithCyclicCenterNeceSuff} we need the
following result.

\begin{lemma}
\label{Lem_TrivialZiZjIntersect}Let $Z$ be an abelian group and let $%
\mathcal{M}:=\left\{ M_{i}\right\} _{i\in I}$ and $\mathcal{N}:=\left\{
N_{j}\right\} _{j\in J}$ be two multisets of subgroups of $Z$. If there
exists an $\mathcal{MN}$-direct factorization system of $Z$, then $M_{i}\cap
N_{j}=\left\{ 1_{G}\right\} $ for any $i\in I$ and $j\in J$.
\end{lemma}

\begin{proof}
Let $\mathcal{A}:=\left\{ A_{i}\right\} _{i\in I}$ and $\mathcal{B}:=\left\{
B_{j}\right\} _{j\in J}$ be an $\mathcal{MN}$-direct factorization system of 
$Z$. Let $i\in I$ and $j\in J$ be arbitrary. Let $g\in M_{i}\cap N_{j}$, $%
a\in A_{i}$ and $b\in B_{j}$. By definition of an $\mathcal{MN}$-direct
factorization system of $Z$, $g\in K\left( A_{i}\right) \cap K\left(
B_{j}\right) $ and therefore $\left( ga\right) b=a\left( gb\right) $ are two
factorizations of an element of $Z$ in $A_{i}\times B_{j}$. By uniqueness $%
ga=a$ implying $g=1_{Z}$. This proves $M_{i}\cap N_{j}=\left\{ 1_{G}\right\} 
$.
\end{proof}

\begin{proof}[Proof of Theorem \protect\ref{Th_DPWithCyclicCenterNeceSuff}]
Apply the notation of Lemma \ref{Lem_Z_M_and_Z_N}(b). Set $\mathcal{M}%
=\left\{ M_{i}\right\} _{i\in I}$ where $M_{i}:=Z_{X_{i}}$ for all $i\in I$,
and $\mathcal{N}=\left\{ N_{j}\right\} _{j\in J}$ where $N_{j}:=Z_{Y_{j}}$
for all $j\in J$. Let $\mathcal{A}=\left\{ A_{i}\right\} _{i\in I}$ where $%
A_{i}=X_{0}$ for all $i\in I$, and $\mathcal{B}=\left\{ B_{j}\right\} _{j\in
J}$ where $B_{j}=Y_{0}$ for all $j\in J$. By Lemma \ref{Lem_Z_M_and_Z_N}(b), 
$M_{i}\subseteq Z_{\left[ M\right] }=\left[ M,M\right] \cap Z$, and by
assumption of the theorem, $Z_{\left[ M\right] }\subseteq K_{Z}\left(
X_{0}\right) $ so $M_{i}\leq K\left( A_{i}\right) $ for each $i\in I$.
Similarly, $N_{j}\leq K\left( B_{j}\right) $ for all $j\in J$. Hence, by
Definition \ref{Def_FactorizationSystem}, $\left( \mathcal{A},\mathcal{B}%
\right) $ is an $\mathcal{MN}$-direct factorization system of $Z$. Now apply
Theorem \ref{Th_ConstructingXxYFromFacSys}. Note that $Z$ itself is an orbit
of conjugacy classes in both $M$ and $N$. Using the notation of Theorem \ref%
{Th_ConstructingXxYFromFacSys}, let $i\in I$ and $j\in J$ be the labels of $%
Z $ as an orbit of conjugacy classes. Choose in the proof of that theorem $%
C_{i}=D_{j}=\left\{ 1_{G}\right\} $. This choice ensures that $X_{0}=X\cap Z$
and $Y_{0}=Y\cap Z$.
\end{proof}

\begin{proposition}
\label{Prop_InducedDPOfMAndN}Let $G$ be a group and $M$ and $N$ normal
subgroups of $G$, such that $G=M\circ _{Z}N$. Let $G=X\times Y$ such that $%
X\subseteq M$ and $Y\subseteq N$. Then 
\begin{equation*}
M=X\times \left( Y\cap Z\right) \text{ and }N=Y\times \left( X\cap Z\right) 
\text{,}
\end{equation*}%
and%
\begin{equation*}
\left[ M,M\right] \cap \left[ N,N\right] =Z_{\left[ M\right] }\cap Z_{\left[
N\right] }=\left\{ 1_{G}\right\} \text{.}
\end{equation*}
\end{proposition}

\begin{proof}
Clearly, $X$ and $\left( Y\cap Z\right) $ are normal subsets of $M$. Since $%
G=X\times Y$ any $m\in M$ has a unique factorization $m=xy$ with $x\in
X\subseteq M$ and $y\in Y\subseteq N$. This implies $y=x^{-1}m\in M$ so $%
y\in Z$. Thus we have proved that every $m\in M$ has a unique factorization $%
m=xy$ with $x\in X$ and $y\in Y\cap Z$, so $M=X\times \left( Y\cap Z\right) $%
. The proof that $N=Y\times \left( X\cap Z\right) $ is similar.

For the proof of the second claim we use the first claim:%
\begin{gather*}
Z_{\left[ M\right] }\cap Z_{\left[ N\right] }\subseteq \left[ M,M\right]
\cap \left[ N,N\right] \\
=\left[ X\times \left( Y\cap Z\right) ,X\times \left( Y\cap Z\right) \right]
\cap \left[ Y\times \left( X\cap Z\right) ,Y\times \left( X\cap Z\right) %
\right] \\
=\left[ X,X\right] \cap \left[ Y,Y\right] \subseteq \left( XX^{-1}\right)
\cap \left( YY^{-1}\right) =\left\{ 1_{G}\right\} \text{,}
\end{gather*}%
where $\left[ X\times \left( Y\cap Z\right) ,X\times \left( Y\cap Z\right) %
\right] =\left[ X,X\right] $ since $Y\cap Z$ is central in $G$, $\left[ X,X%
\right] \subseteq \left( XX^{-1}\right) $ since $X$ is normal, and the last
equality follows from Lemma \ref{Lem_EquivDirectnessConditions}.
\end{proof}

\begin{proof}[Proof of Corollary \protect\ref%
{Coro_PrimePowerWithNoFixedPoints}]
Set $M:=G$ and $N:=Z:=\left\langle z\right\rangle $. Then $G=M\circ _{Z}N$
and by assumption $o\left( z\right) =p^{k}$ where $p$ is a prime and $k\geq
2 $ an integer. Let $X_{0}=\left\langle z^{p}\right\rangle $ and $%
Y_{0}=\left\{ 1_{Z},z,z^{2},...,z^{p-1}\right\} $. Then $X_{0}$ is a
subgroup of $Z$ of order $p^{k-1}$ and $Y_{0}$ is a transversal of $X_{0}$
in $Z$ whereby $Z=X_{0}\times Y_{0}$. It is immediate to check that $%
K_{Z}\left( X_{0}\right) =X_{0}$ and $K_{Z}\left( Y_{0}\right) =\left\{
1_{Z}\right\} $. Consider the multiplication action of $Z$ on $\Omega _{G}$.
Since \ $z$ is semi-regular, the stabilizer of any conjugacy class is proper
in $Z$. But $X_{0}$ is the unique maximal subgroup of $Z$ and hence it
contains all of the point stabilizers. By Lemma \ref{Lem_Z_M_and_Z_N}(b)\
this implies\ $Z_{\left[ M\right] }=\left[ M,M\right] \cap Z\subseteq
X_{0}=K_{Z}\left( X_{0}\right) $. The condition $\left[ N,N\right] \cap
Z\subseteq K_{Z}\left( Y_{0}\right) $ is immediate to verify. Thus all the
conditions of Theorem \ref{Th_DPWithCyclicCenterNeceSuff} are satisfied and
we can deduce the existence of a set-direct factorization $G=X\times Y$ with 
$X\subseteq M$, $Y\subseteq N$, $X\cap Z=X_{0}$ and $Y\cap Z=Y_{0}$, which
is non-trivial since $Z=X_{0}\times Y_{0}$ is non-trivial. Observe that both 
$Y=Y_{0}$ and $X$ are normalized, and that $Y$ is not a group. Now assume
that $G$ is perfect. We will show that under this assumption also $X$ is not
a group. Suppose by contradiction that $X$ is a group. Since $G=X\times Y$
and $Y$ is central, we have $\left[ G,G\right] =\left[ X,X\right] $ and
hence $X=\left\langle X\right\rangle $ contains all commutators in $G$ and
hence also the derived subgroup $G^{\prime }=\left\langle \left[ G,G\right]
\right\rangle $. But $G^{\prime }=G$ by assumption and hence $X=G$. This
contradicts $Y=Y_{0}\nsubseteq X$.
\end{proof}

\section{Finite quasi-simple groups}

In this section we prove Theorem \ref{Th_QuasiSimple} which states
conditions for the existence of non-trivial set-direct decompositions of
finite quasi-simple groups. The proof demonstrates the use of some of the
results of the previous sections. Recall that a group $G$ is quasi-simple if 
$G=G^{\prime }$ and $G/Z\left( G\right) $ is a simple group. The center of a
finite quasi-simple group $G$ must be isomorphic to a factor group of the
Schur multiplier of $G/Z\left( G\right) $ (\cite[Section 33]{Aschbacher}).
Information on the relevant Schur multipliers is given in \cite{WebAtlas}.

\begin{lemma}
\label{Lem_NonTrivialQuasi}Let $G$ be a finite quasi-simple group such that $%
G=X\times Y$ with $\left\vert X\right\vert >1$ and $\left\vert Y\right\vert
>1$. Then precisely one of $X$ and $Y$ is central and the other factor,
which has non-central elements, generates $G$ but is not a group.
\end{lemma}

\begin{proof}
Set $M:=\left\langle X\right\rangle $, $N:=\left\langle Y\right\rangle $ and 
$Z:=M\cap N$. By Theorem \ref{Th_MainNecessaryAndSufficient}, $G=M\circ
_{Z}N $ and $Z\leq Z\left( G\right) $. We have $G/Z=\left( M/Z\right) \times
\left( N/Z\right) $ and both $M/Z$ and $N/Z$ are perfect groups. Since $%
G/Z\left( G\right) \cong \left( G/Z\right) /\left( Z\left( G\right)
/Z\right) $ is simple, we must have that one of $M/Z$ and $N/Z$ is trivial.
Assume without loss of generality that $N=Z$. Then $Y$ is central and $M=G$.
Now suppose by contradiction that $X$ is a group. Then, since $Y$ is central
we get $\left[ G,G\right] =\left[ X,X\right] $ (see the first paragraph of
Section \ref{Sect_NotationBackground}). But this gives, using the fact that $%
G$ is perfect%
\begin{equation*}
G=G^{\prime }=\left\langle \left[ X,X\right] \right\rangle =X^{\prime }\leq X%
\text{,}
\end{equation*}%
in contradiction to $\left\vert Y\right\vert >1$.
\end{proof}

Let $G$ be a finite quasi-simple group. Blau \cite{Blau1994} has essentially
classified the semi-regular elements of all finite quasi-simple groups. The
following rephrasing of \cite[Theorem 1]{Blau1994} is a key result for the
proof of Theorem \ref{Th_QuasiSimple}. Recall that $z\in Z\left( G\right) $
is semi-regular if $zC\neq C$ for any $C\in \Omega _{G}$.

\begin{theorem}[{\protect\cite[Theorem 1]{Blau1994}}]
\label{Th_Blau}Let $G$ be a finite quasi-simple group. If $z\in Z\left(
G\right) $ is semi-regular then one of the following holds:
\end{theorem}

\begin{enumerate}
\item[(i)] $G/Z(G)=A_{6}$, $A_{7}$, $\mathrm{Fi}_{22}$, $\mathrm{PSU}(6,4)$,
or $^{2}\mathrm{E}_{6}(4)$ and $o\left( z\right) =6$.

\item[(ii)] $G/Z(G)=\mathrm{PSU}(4,9)$, $\mathrm{M}_{22}$, or $G/Z(G)=%
\mathrm{PSL}(3,4)$ with $Z\left( G\right) $ cyclic, and $o\left( z\right)
\in \left\{ 6,12\right\} $.

\item[(iii)] $G/Z(G)=\mathrm{PSL}(3,4)$, $Z\left( G\right) $ is non-cyclic
and $o\left( z\right) \in \left\{ 2,4,6,12\right\} $.
\end{enumerate}

For case (iii) of the last theorem we need more detailed information. We use
the fact that the Schur multiplier of $\mathrm{PSL}(3,4)$ is isomorphic to $%
C_{3}\times C_{4}\times C_{4}$.

\begin{lemma}
\label{Lem_PSL34_Non_Cyclic}Let $G$ be a finite quasi-simple group with $%
G/Z(G)=\mathrm{PSL}(3,4)$, and $Z(G)$ is non-cyclic. Then:

\begin{description}
\item[(a)] All elements of $Z\left( G\right) $ which are of order $6$ or $12$
are semi-regular.

\item[(b)] $Z\left( G\right) $ has a semi-regular $2$-element if and only if 
$\left\vert Z\left( G\right) \right\vert $ is divisible by $8$.

\item[(c)] If $\left\vert Z\left( G\right) \right\vert $ is divisible by $8$
but not by $16$ then $Z\left( G\right) $ has precisely one semi-regular
involution and no semi-regular element of order $4$.

\item[(d)] If $\left\vert Z\left( G\right) \right\vert $ is divisible by $16$
then $Z\left( G\right) $ contains precisely six semi-regular elements of
order $4$ and no semi-regular involution.
\end{description}
\end{lemma}

\begin{proof}
$Z\left( G\right) $ must be isomorphic to one of the following six
non-cyclic subgroups of $C_{3}\times C_{4}\times C_{4}$: 
\begin{equation*}
C\times C_{4}\times C_{4},~C\times C_{4}\times C_{2},~C\times C_{2}\times
C_{2}\text{,}
\end{equation*}%
where $C$ is either the trivial group or a group of order $3$. Moreover,
each of these six groups determines a unique, up to isomorphism,
quasi-simple $G$ with $G/Z(G)=\mathrm{PSL}(3,4)$. By \cite[Lemma 1]{Blau1994}%
, $z\in Z\left( G\right) $ is semi-regular if and only if $%
\tsum\limits_{\chi \in \mathrm{Irr}\left( G\right) }\chi \left( z\right)
/\chi \left( 1\right) =0$, where $\mathrm{Irr}\left( G\right) $ is the set
of complex irreducible characters of $G$. The character tables of all of the
six relevant groups are implemented in GAP's character table library (\cite%
{GAP2017},\cite{GAPCTblLib1.2.1}), and this was used in order to identify
the central elements and check, for each one of them, the cited condition.
\end{proof}

The following two lemmas are needed for analyzing the case of a quasi-simple 
$G$ with $G/Z(G)=\mathrm{PSL}(3,4)$.

\begin{lemma}
\label{Lem_NonDirectH6}Let $H_{6}=\left\langle a\right\rangle \times
\left\langle b\right\rangle $ with $o\left( a\right) =3$ and $o\left(
b\right) =2$. Let $X:=\left\{ 1,x\right\} $ where $x\neq 1$ is some element
of $H_{6}$ and $Y:=\left\{ 1,a,ab\right\} $. Then $XY$ is not direct.
\end{lemma}

\begin{proof}
A direct computation gives $YY^{-1}=H_{6}$. Hence $XX^{-1}\cap
YY^{-1}=XX^{-1}$. Since $x\in XX^{-1}$, we get $XX^{-1}\neq \left\{
1_{G}\right\} $, and therefore the claim follows from Lemma \ref%
{Lem_EquivDirectnessConditions}.
\end{proof}

\begin{lemma}
\label{Lem_NonDirectH12}Let $H_{12}=\left\langle a\right\rangle \times
\left\langle b\right\rangle $ with $o\left( a\right) =3$ and $o\left(
b\right) =4$. Let $X\subseteq H_{12}$ where $1\in X$ and $\left\vert
X\right\vert =4$, and let $Y:=\left\{ 1,b,ab^{\varepsilon }\right\} $ with $%
\varepsilon \in \left\{ 1,-1\right\} $ . Then $XY$ is not direct.
\end{lemma}

\begin{proof}
Assume, by contradiction, that $XY$ is direct. Then, by Lemma \ref%
{Lem_EquivDirectnessConditions}, $XX^{-1}\cap YY^{-1}=\left\{ 1\right\} $. \
We have%
\begin{equation*}
YY^{-1}=\left\{ 1,b,b^{-1}\right\} \cup S_{Y}\text{ where }S_{Y}:=\left\{
ab^{\varepsilon },a^{-1}b^{-\varepsilon },a^{-1}b^{1-\varepsilon
},ab^{\varepsilon -1}\right\} \text{.}
\end{equation*}%
Since $\left\vert YY^{-1}\right\vert =7$ and $|H_{12}|=12$, the size of $%
XX^{-1}$ is either $4$, $5$, or $6$. Also note that (using $b^{2-\varepsilon
}=b^{\varepsilon }$ for all $\varepsilon \in \left\{ 1,-1\right\} $): 
\begin{equation*}
H_{12}\backslash \left\{ YY^{-1}\right\} =\left\{ b^{2}\right\} \cup S_{X}%
\text{ where }S_{X}:=b^{2}S_{Y}=\left\{ ab^{-\varepsilon
},a^{-1}b^{\varepsilon },a^{-1}b^{-\left( 1+\varepsilon \right)
},ab^{1+\varepsilon }\right\} \text{.}
\end{equation*}

Since $\{ 1 \} = X \cap YY^{-1}$, the set $X$ has no element of order $4$
and hence $X$ is not a subgroup.

Assume that $|XX^{-1}|=4$. Since $X\cup X^{-1}\subseteq XX^{-1}$, we must
have $XX^{-1}=X\cup X^{-1}$ and so $X=X^{-1}$ implying that $X^{2}=X$. Thus $%
X$ is a subgroup contradicting our previous assertion.

Assume that $\left\vert XX^{-1}\right\vert =5$. Both $X\cup X^{-1}$ and $%
XX^{-1}$ are inverse closed and contain $1$. Also, for any $h\in H_{12}$, $%
h=h^{-1}$ if and only if \ $h\in \left\{ 1,b^{2}\right\} $. Since $%
\left\vert XX^{-1}\right\vert $ is odd this implies $b^{2}\notin XX^{-1}$.
But then $b^{2}\notin X\cup X^{-1}$ and hence $\left\vert X\cup
X^{-1}\right\vert $ is odd, implying $XX^{-1}=X\cup X^{-1}=\{1\}\cup S_{X}$.

Observe that for any $\delta \in \{1,-1\}$, we have $a^{-1}b^{\varepsilon
}\in X^{\delta }$ if and only if $ab^{1+\varepsilon }\in X^{\delta }$, since
otherwise $b^{-1} = \left( a^{-1}b^{\varepsilon }\right) \left(
ab^{1+\varepsilon }\right) \in XX^{-1}$ is a contradiction. Similarly, $%
ab^{-\varepsilon }\in X^{\delta }$ if and only if $a^{-1}b^{-\left(
1+\varepsilon \right) }\in X^{\delta }$. But the combination of the last two
assertions contradicts $\left\vert X\right\vert =4$.

Assume that $|XX^{-1}|=6$. Then $XX^{-1}=\{1,b^{2}\}\cup S_{X}$. First
suppose that $b^{2}\not\in X$. Then $|X\cup X^{-1}|=5$ as in the discussion
of the $\left\vert XX^{-1}\right\vert =5$ case, and this implies $X\cup
X^{-1}=\left\{ 1\right\} \cup S_{X}$. But now a routine check shows that $%
b^{2}\notin \left( X\cup X^{-1}\right) ^{2}\supseteq XX^{-1}$ - a
contradiction. Thus $b^{2}\in X$ and hence $b^{2}X^{-1}\subseteq XX^{-1}$.
Since $b^{2}S_{X}=S_{Y}\subseteq YY^{-1}$, we have 
\begin{equation*}
b^{2}(X^{-1}\cap S_{X})=b^{2}X^{-1}\cap b^{2}S_{X}\subseteq XX^{-1}\cap
YY^{-1}=\{1\}\text{.}
\end{equation*}%
Since $b^{2}$ is an involution, it follows that $X^{-1}\cap S_{X}\subseteq
\{b^{2}\}$, implying $X^{-1}\cap S_{X}=\emptyset $ and $\left\vert
X\right\vert =2$ - the final contradiction.
\end{proof}

\begin{proof}[Proof of Theorem \protect\ref{Th_QuasiSimple}]
Suppose that $G=X\times Y$ such that $X$ and $Y$ are normalized and $%
\left\vert X\right\vert >1$ and $\left\vert Y\right\vert >1$. Set $%
M=\left\langle X\right\rangle $, $N=\left\langle Y\right\rangle $, $Z:=M\cap
N$. By Theorem \ref{Th_MainNecessaryAndSufficient}, $G=M\circ _{Z}N$, and by
Lemma \ref{Lem_NonTrivialQuasi} we can assume $N=Z=\left\langle
Y\right\rangle $ and $M=G$. Since $Y$ is normalized we have $1\in Y$, and
since $1\in M_{m}$ for every $m\in M$ we get, by Theorem \ref%
{Th_FacSysProperties}(a), that for all $1\neq y\in Y$ and all $m\in M$, $%
y\notin M_{m}$. Thus, all non-trivial elements of $Y$ must be semi-regular,
and it suffices to consider groups $G$ which fall into one of the cases
(i)-(iii) of Theorem \ref{Th_Blau}. We now split the discussion into three
cases.

\begin{description}
\item[Case I] $G/Z\left( G\right) $ is not isomorphic to $\mathrm{PSL}(3,4)$
or $G/Z\left( G\right) =\mathrm{PSL}(3,4)$ and $8$ does not divide $%
\left\vert Z\left( G\right) \right\vert $. We have to show that the
assumptions of Case I lead to a contradiction. By Theorem \ref%
{Th_XxYAndFactorizationSystems}, using its notation, there exists an
associated $\mathcal{MN}$-direct factorization system $\left( \left\{
X_{m}\right\} _{m\in M},\left\{ Y_{n}\right\} _{n\in N}\right) $ of $Z$. It
suffices to show that the assumption $Z=X_{m}\times Y$, for all $m\in M$,
leads to a contradiction. By Theorem \ref{Th_Blau} and by Lemma \ref%
{Lem_PSL34_Non_Cyclic}, all non-trivial $y\in Y$ satisfies $o\left( y\right)
\in \left\{ 6,12\right\} $. Since $Y\subseteq Z$ this implies that $%
\left\vert Z\right\vert $ is divisible by $6$. Therefore, $Z$ has an element
of order $2$ and an element of order $3$, and each such element must fix at
least one conjugacy class of $G$, whereby we must have $m_{1},m_{2}\in M$
(not necessarily distinct!) such that $2|\left\vert M_{m_{1}}\right\vert $
and $3|\left\vert M_{m_{2}}\right\vert $. Now, by Equations \ref%
{Eq_lA_i|AreAllEqual} and \ref{Eq_LcmCondition}, $6$ divides $\left\vert
X_{m}\right\vert $ for all $m\in M$. Hence, since $\left\vert Y\right\vert
>1 $ we can conclude that $\left\vert Z\right\vert =\left\vert
X_{m}\right\vert \left\vert Y\right\vert \geq 12$ (for any $m\in M$).
Suppose that $\left\vert Z\right\vert =12$. In this case $\left\vert
Y\right\vert =2$ and since $Z=N=\left\langle Y\right\rangle $ and $1\in Y$,
the single non-trivial element of $Y$ must be a generator of $Z$ and hence
has order $12 $. Therefore $Z$ has an element of order $4$, which fixes a
conjugacy class, and Equation \ref{Eq_LcmCondition} implies that $12$
divides $\left\vert X_{m}\right\vert $ in contradiction to $\left\vert
Y\right\vert =2$. Thus, we must have $\left\vert Z\right\vert >12$.
Examining the relevant Schur multipliers of $G/Z\left( G\right) $, we are
left with the possibility $G/Z\left( G\right) =\mathrm{PSU}(4,9)$. The Schur
multiplier of this group is isomorphic to $C_{3}\times C_{3}\times C_{4}$.
The condition $\left\vert Z\right\vert >12$ and the fact that $\left\vert
Z\right\vert $ divides the order of the Schur multiplier leave two
possibilities for $Z$.

\begin{description}
\item[(1)] $Z\cong C_{3}\times C_{3}\times C_{2}$. Choose generators for the
direct factors so that $Z=\left\langle g_{1}\right\rangle \times
\left\langle g_{2}\right\rangle \times \left\langle g_{3}\right\rangle $
with $o\left( g_{1}\right) =o\left( g_{2}\right) =3$ and $o\left(
g_{3}\right) =2$. We have $Z=\left\langle Y\right\rangle $, and all
non-trivial elements of $Y$ have order $6$. As argued above, $6~|~\left\vert
X_{m}\right\vert $ for all $m\in M$, and hence $6\left\vert Y\right\vert ~$%
divides$~\left\vert X_{m}\right\vert \left\vert Y\right\vert =\left\vert
Z\right\vert =18$, forcing $\left\vert Y\right\vert =3$. Thus $Y=\left\{
1,y_{1},y_{2}\right\} $ with $o\left( y_{1}\right) =o\left( y_{2}\right) =6$%
. Note that $y_{i}=\theta _{i}g_{3}$ where $\theta _{1}$ and $\theta _{2}$
are elements of order $3$. Since $Z=\left\langle Y\right\rangle $ we get
that $\theta _{1}$ and $\theta _{2}$ belong to two distinct order $3$
subgroups. Note that $Z$ has $4$ distinct subgroups of order $3$, and fixing
any two of them, the other two are the two diagonal subgroups of the direct
product of the first two. Therefore we can assume, without loss of
generality, that $y_{1}:=g_{1}g_{3}$ and $y_{2}:=g_{2}g_{3}$. Observe that $%
g_{1}$, being of order $3$, must fix a conjugacy class of $G$. It follows
that there exists some $m\in M$ such that $g_{1}$ fixes the conjugacy class
of $m$. Thus $X_{m}g_{1}$ $=X_{m}$ and so $\left\langle g_{1}\right\rangle
\leq K\left( X_{m}\right) $. However, we will now prove that there is no $%
m\in M$ such that $Z=X_{m}\times Y$ and $\left\langle g_{1}\right\rangle
\leq K\left( X_{m}\right) $. Assume by contradiction that there exists $m\in
M$ such that $Z=X_{m}\times Y$ and $\left\langle g_{1}\right\rangle \leq
K\left( X_{m}\right) $. Note that the cyclic group $\left\langle
g_{2}g_{3}\right\rangle $ of order $6$ is a transversal of $\left\langle
g_{1}\right\rangle $, and hence $X_{m}=\left\langle g_{1}\right\rangle
\alpha _{1}\cup \left\langle g_{1}\right\rangle \alpha _{2}$, where $\alpha
_{1}$ and $\alpha _{2}$ are two distinct elements of $\left\langle
g_{2}g_{3}\right\rangle $. By Lemma \ref%
{Lem_TrivialCentersImplyIdentityInside}(c) $Z=\left( \alpha
_{1}^{-1}X_{m}\right) \times Y$, and hence we can assume, without loss of
generality, that $X_{m}=\left\langle g_{1}\right\rangle \cup \left\langle
g_{1}\right\rangle \alpha $ for some $\alpha \notin \left\langle
g_{1}\right\rangle $. Furthermore, $Z=X_{m}\times Y$ implies $Z/\left\langle
g_{1}\right\rangle =\left\{ 1,\alpha \right\} \times \left\{
1,g_{2},g_{2}g_{3}\right\} $, where, by a slight abuse of notation, we label 
$\left\langle g_{1}\right\rangle $ cosets of $Z$ by their $Z$%
-representatives names. This contradicts the assertion of Lemma \ref%
{Lem_NonDirectH6}. Hence the possibility $Z\cong C_{3}\times C_{3}\times
C_{2}$ is ruled out By Theorem \ref{Th_XxYAndFactorizationSystems} (see also
Definition \ref{Def_FactorizationSystem}).

\item[(2)] $Z\cong C_{3}\times C_{3}\times C_{4}$. Choose generators for the
direct factors so that $Z=\left\langle g_{1}\right\rangle \times
\left\langle g_{2}\right\rangle \times \left\langle g_{4}\right\rangle $
with $o\left( g_{1}\right) =o\left( g_{2}\right) =3$ and $o\left(
g_{4}\right) =4$. In this case, every element in $Z$ of order $3$ or $4$
must fix a conjugacy class of $G$. Hence, by Equation \ref{Eq_LcmCondition}, 
$12$ divides $\left\vert X_{m}\right\vert $ for all $m\in M$. By similar
arguments to those of case (1), $Y=\left\{ 1,y_{1},y_{2}\right\} $ with $%
\left\{ o\left( y_{1}\right) ,o\left( y_{2}\right) \right\} \subseteq
\left\{ 6,12\right\} $. If $o\left( y_{1}\right) =o\left( y_{2}\right) =6$
then $\left\langle Y\right\rangle <Z$ - a contradiction. Suppose that $%
Y=\left\{ 1,y_{1},y_{2}\right\} $ with $o\left( y_{1}\right) =6$, $o\left(
y_{2}\right) =12$ and $\left\langle Y\right\rangle =Z$. \ Then (compare with
the treatment of the previous case $Z\cong C_{3}\times C_{3}\times C_{2}$)
we may assume that $y_{1}:=g_{1}g_{4}^{2}$ and $y_{2}:=g_{2}g_{4}$ or $%
y_{1}:=g_{2}g_{4}^{2}$ and $y_{2}:=g_{1}g_{4}$. Assume $%
y_{1}:=g_{1}g_{4}^{2} $ and $y_{2}:=g_{2}g_{4}$. Then $Yy_{2}^{-1}=\left\{
1,y_{1}^{\prime }:=y_{1}y_{2}^{-1}=g_{1}g_{2}^{-1}g_{4},y_{2}^{\prime
}:=y_{2}^{-1}\right\} $ and $o\left( y_{1}^{\prime }\right) =o\left(
y_{2}^{\prime }\right) =12$. By Lemma \ref%
{Lem_TrivialCentersImplyIdentityInside}(c) $Z=X_{m}\times \left(
Yy_{2}^{-1}\right) $. Applying similar considerations for the case $%
y_{1}:=g_{2}g_{4}^{2}$ and $y_{2}:=g_{1}g_{4}$, we conclude that it remains
to consider $Y=\left\{ 1,y_{1},y_{2}\right\} $ with $o\left( y_{1}\right)
=o\left( y_{2}\right) =12$. We prove, by contradiction, that there is no $%
m\in M$ such that $Z=X_{m}\times Y$ and $\left\langle g_{1}\right\rangle
\leq K\left( X_{m}\right) $. Arguing as before, we can assume, without loss
of generality, that $y_{1}:=g_{1}g_{4}^{\varepsilon _{1}}$ and $%
y_{2}:=g_{2}g_{4}^{\varepsilon _{2}}$ with $\varepsilon _{1},\varepsilon
_{2}\in \left\{ -1,1\right\} $. Using the fact that $Z=X_{m}\times Y$ if and
only if $Z=X_{m}^{-1}\times Y^{-1}$ implies that it would suffice to check $%
\left( \varepsilon _{1},\varepsilon _{2}\right) =\left( 1,1\right) $\ and $%
\left( \varepsilon _{1},\varepsilon _{2}\right) =\left( 1,-1\right) $. The
cyclic group $\left\langle g_{2}g_{4}\right\rangle $ of order $12$ is a
transversal of $\left\langle g_{1}\right\rangle $, and hence we can assume,
without loss of generality, using arguments as before, that $%
X_{m}=\left\langle g_{1}\right\rangle \cup \left\langle g_{1}\right\rangle
\alpha _{1}\cup \left\langle g_{1}\right\rangle \alpha _{2}\cup \left\langle
g_{1}\right\rangle \alpha _{3}$, where $\alpha _{1},\alpha _{2},\alpha _{3}$
represent any three distinct non-trivial cosets of $\left\langle
g_{1}\right\rangle $ in $Z$ which are contained in $\left\langle
g_{2}g_{4}\right\rangle \left\langle g_{1}\right\rangle \backslash
\left\langle g_{1}\right\rangle Y$. Furthermore, $Z=X_{m}\times Y$ implies $%
Z/\left\langle g_{1}\right\rangle =\left\{ 1,\alpha _{1},\alpha _{2},\alpha
_{3}\right\} \times \left\{ 1,g_{4},g_{2}g_{4}^{\pm 1}\right\} $. This
contradicts the assertion of Lemma \ref{Lem_NonDirectH12}. Hence the
possibility $Z\cong C_{3}\times C_{3}\times C_{4}$ is ruled out.
\end{description}

\item[Case II] $G/Z(G)=\mathrm{PSL}(3,4)$, and $Z(G)$ is non-cyclic such
that $8$ divides $\left\vert Z\left( G\right) \right\vert $ but $16$ does
not divide $\left\vert Z\left( G\right) \right\vert $. By Lemma \ref%
{Lem_PSL34_Non_Cyclic} (c), $Z(G)$ has a unique semi-regular involution $z$.
Choose $Z=\left\langle z\right\rangle \cong C_{2}$. Clearly $G=M\circ _{Z}N$
with $M=G$ and $N=Z$. Notice that each orbit of conjugacy classes of $G$
with respect to the multiplication action of $Z$ has length $2$. Setting $%
I:=O\left( \Omega _{M}\right) $ and $J:=O\left( \Omega _{N}\right) $, we
have $\left\vert J\right\vert =1$, and all of the point stabilizers, $M_{i}$
with $i\in I$ and the single $N_{j}$ with $j\in J$ are trivial. Thus $%
\mathcal{A}:=\left\{ A_{i}\right\} _{i\in I}$ and $\mathcal{B}:=\left\{
B_{j}\right\} _{j\in J}$ is an $\mathcal{MN}$-direct factorization system of 
$Z$, where $\mathcal{M}=\left\{ M_{i}\right\} _{i\in I}$, $\mathcal{N}%
=\left\{ N_{j}\right\} _{j\in J}$, $A_{i}=\left\{ 1\right\} $ for all $i\in
I $ and $B_{j}=Z$. By Theorem \ref{Th_ConstructingXxYFromFacSys} we get a
normalized set-direct factorization $G=X\times Z$, where $X$ contains
precisely one conjugacy class from each $Z$-orbit, and $X$ is not a group by
Lemma \ref{Lem_NonTrivialQuasi}.

\item[Case III] $G/Z(G)=\mathrm{PSL}(3,4)$, and $Z(G)$ is non-cyclic such
that $16$ divides $\left\vert Z\left( G\right) \right\vert $. By Lemma \ref%
{Lem_PSL34_Non_Cyclic} (d), $Z(G)$ has semi-regular elements of order $4$.
In this case $G$ satisfies all the conditions of Corollary \ref%
{Coro_PrimePowerWithNoFixedPoints} and hence $G$ has a non-trivial
normalized set-direct factorization such that none of the factors is a group.
\end{description}
\end{proof}

\section{Products of two conjugacy classes}

In this section we consider the directness of products of two conjugacy
classes.

\begin{theorem}
\label{Th_CDNonDirect}Let $G$ be a group having a unique minimal normal
subgroup $N$. Suppose in addition that $N$ is non-abelian. Let $C$ and $D$
be any two non-trivial conjugacy classes of $G$. Then the product $CD$ is
non-direct. In particular, if $G$ is a finite almost simple group then the
product of any two non-trivial conjugacy classes of $G$ is non-direct.
\end{theorem}

\begin{proof}
Suppose by contradiction that the product $CD$ is direct. By Theorem \ref%
{Th_DirectImpliesCentralizing}, $C$ and $D$ centralize each other. Now $%
\left\langle C\right\rangle $ and $\left\langle D\right\rangle $ are
non-trivial normal subgroups of $G$, and hence $N\leq \left\langle
C\right\rangle \cap \left\langle D\right\rangle $. But since $C$ and $D$
centralize each other we have that $\left\langle C\right\rangle \cap
\left\langle D\right\rangle $ is abelian while\ $N$ is not \ - a
contradiction.
\end{proof}

\begin{remark}
The conclusion of Theorem \ref{Th_CDNonDirect} is false if $G$ has more than
one non-abelian minimal normal subgroup. In fact, let $G$ be a group and let 
$N_{1}$ and $N_{2}$ be two minimal normal subgroups of $G$ which are not
central. Then there are two non-trivial conjugacy classes $C_{1}$ and $C_{2}$
of $G$, such that $C_{1}\subset N_{1}$ and $C_{2}\subset N_{2}$ and it is
easy to see that $C_{1}C_{2}$ is direct. Furthermore, the conclusion of
Theorem \ref{Th_CDNonDirect} need not be true if $G$ has an abelian minimal
normal subgroup which is non-central. Consider, for example, a dihedral
group of order $10$. Then $G$ has a unique minimal normal subgroup $%
\left\langle g\right\rangle $ of order $5$. Set $C_{1}:=\left\{
g,g^{4}\right\} $ and $C_{2}:=\left\{ g^{2},g^{3}\right\} $. These are
conjugacy classes of $G$ and their product is direct.
\end{remark}

\bigskip

\centerline{\textbf{Acknowledgement}} 

\medskip

We are indebted to the anonymous referee for noticing a mistake in a previous version of Theorem \ref{Th_QuasiSimple} and for the many comments which significantly improved the correctness and clarity of our original exposition. We are also grateful
to S. Szab\'{o} for a helpful communication.

\bigskip

\bigskip

\end{document}